\documentclass[11pt]{amsart}
\usepackage{amsmath,amsfonts,amscd,amssymb}
\usepackage[margin=2.9cm]{geometry}
\theoremstyle{plain}
\newcounter{thmcount}

\newtheorem{theorem}[thmcount]{Theorem}

\newtheorem{lemma}[thmcount]{Lemma}

\theoremstyle{definition}
\newtheorem{remark}[thmcount]{Remark}
\newtheorem*{remark*}{Remark}

\newtheorem{definition}[thmcount]{Definition}

\usepackage[utf8]{inputenc}
\usepackage{amssymb}
\usepackage[english]{babel}

\usepackage[OT2,T1]{fontenc}
\DeclareSymbolFont{cyrletters}{OT2}{wncyr}{m}{n}
\DeclareMathSymbol{\Sha}{\mathalpha}{cyrletters}{"58}

\usepackage{amsthm}
\usepackage{mathtools}

\usepackage[alphabetic,lite]{amsrefs}

\usepackage{hyperref}

\mathtoolsset{showonlyrefs}

\title{A note on the growth of Sha in dihedral extensions}


\author{Jamie Bell}
\address{University College London, Gower Street, London, WC1E 6BT, UK}
\email{james.bell.20@ucl.ac.uk}

\begin{document}

\begin{abstract}
    We provide a formula for the order of the Tate--Shafarevich group of elliptic curves over dihedral extensions of number fields of order $2n$, up to $4^{th}$ powers and primes dividing $n$. Specifically, for odd $n$ it is equal to the order of the Tate--Shafarevich group over the quadratic subextension. A similar formula holds for even $n$.
\end{abstract}

\maketitle

\section*{Introduction}
Computing the Tate--Shafarevich group of an elliptic curve is a difficult problem. It is conjectured to be finite, though this is not known for curves of analytic rank greater than 1. The main result about its size is due to Cassels, who proved that if it is finite, it must have square order \cite{Cassels}. This has some applications, for example to the parity conjecture.

In this paper, we show that under some circumstances we can determine $\Sha$ modulo $4^{th}$ powers in terms of $\Sha$ over smaller number fields. More precisely, we prove the following:

\begin{theorem}\label{maintheorem}
    Let $E$ be an elliptic curve over a number field $k$ and $F/k$ a dihedral extension of degree $2n$.  Let $p$ be a prime not dividing $n$ and assume $\Sha(E/F)[p^{\infty}]$ is finite \footnote{Note that $\Sha[p^{\infty}]$ is finite over subfields of $F$ by \cite{DD10} Remark 2.10.}. Then
    \begin{itemize}
        \item if $n$ is odd, and $K$ is the quadratic extension of $k$ contained in $F$, then there is an integer $t$ such that $$\frac{|\Sha(E/F)[p^{\infty}]|}{|\Sha(E/K)[p^{\infty}]|} = p^{4t}.$$ 
        \item if $n$ is even, and $K_1$, $K_2$ and $K_3$ are the three quadratic extensions of $k$ contained in $F$, then there is an integer $t$ such that
        $$\frac{|\Sha(E/F)[p^{\infty}]|}{|\Sha(E/K_1)[p^{\infty}]||\Sha(E/K_2)[p^{\infty}]||\Sha(E/K_3)[p^{\infty}]|} =p^{4t}.$$
    \end{itemize}
    
\end{theorem}

\begin{remark}
    Although the conclusion for odd $n$ compares two Tate--Shafarevich groups in a cyclic extension $F/K$, we do need this to be contained in a dihedral extension to get the result -- it is not true for all cyclic extensions. For example, over the extension of $\mathbb{Q}$ given by adjoining a root of $x^3 - x^2 - 2x + 1$, which is cyclic of degree 3, consider the curve $y^2+y=x^3-x^2-10x-20$ (11a1). Using Magma (\cite{Magma}), we can calculate it has analytic order of $\Sha$ equal to 25, whereas it has trivial $\Sha$ over $\mathbb{Q}$. Therefore the $5^{\infty}$-part of $\Sha$ changes by a non-fourth-power.
\end{remark}
\begin{remark}
    The result also does not hold if $p|n$. For example, take the curve $y^2+xy+y=x^3-351233x-80149132$ (210b7) over the number field given by adjoining a root of $x^{6} - 30x^{4} + 225x^{2} - 200$, which is a $D_6$-extension of $\mathbb{Q}$. The analytic order of $\Sha$ is 4 over the quadratic extension and $2^{10}3^6$ over the $D_6$ extension (\cite{Magma}). This is equal to the order of $\Sha$ over the quadratic extension times $2^83^6$, so the $3^{\infty}$-part changes by a non-fourth-power.
\end{remark}

We can also say something about the Galois--module structure of $\Sha$.

\begin{theorem}\label{theorem2}
    Suppose $E, k, F, n$ and $p$ are as in Theorem \ref{maintheorem}, and $K$ is a quadratic subextension of $F/k$ with $H \cong C_n$ the Galois group of $F/K$. Assume that $\Sha(E/F)[p^{\infty}]$ is finite. Then as $\mathbb{Z}_p[H]$--modules, $\Sha(E/F)[p^{\infty}] \cong X \oplus X$ for some submodule $X$, and $\Sha(E/K)[p^{\infty}] \cong X^H \oplus X^H$.

    Moreover, when $n$ is odd and $p^a \equiv -1 \mod{n}$ has a solution, $|X|/|X^H|$ is a square.
\end{theorem}

Note that the second part of Theorem \ref{theorem2} reproves Theorem \ref{maintheorem} in this case.

\subsection*{Notation}
Let $W^K_k(E)$ be the Weil restriction of $E/K$ to $k$.

Let $C_n$ and $D_{2n}$ denote the cyclic group of order $n$ and the dihedral group of order $2n$.

Let $[n]$ denote the multiplication by $n$ isogeny on an elliptic curve, or the multiplication by $n$ map on an abelian group. Let $G[n]$ be the kernel of $[n]$ on $G$, and $G[n^{\infty}]$ the union of $G[n^k]$ for positive integers $k$.

Let $\Sha(A/k)$ denote the Tate--Shafarevich group of an abelian variety $A/k$.

Let $\mathbb{Z}_{(p)}$ be the localisation of $\mathbb{Z}$ at $p$.

Let $F/k$ be a dihedral extension of number fields of degree $2n$, with quadratic subextensions $K$ (for odd $n$) or $K_1, K_2$ and $K_3$ (for even $n$), as in the statement of Theorem $\ref{maintheorem}$.

\subsection*{Acknowledgements}
I would like to thank my supervisor Vladimir Dokchitser for guiding my research towards this observation, and his essential help in sorting out the details of the proof.

This work was supported by the Engineering and Physical Sciences Research Council [EP/L015234/1], the EPSRC Centre for Doctoral Training in Geometry and Number Theory (The London School of Geometry and Number Theory) at University College London.

\section*{The size of $\Sha$}
\begin{definition}
    Recall that a Brauer relation in a group $G$ is a formal sum of subgroups $\sum_i H_i - \sum_j H'_j$ satisfying $\bigoplus_i \mathbb{Q}[G/H_i] \cong \bigoplus_j \mathbb{Q}[G/H_j']$ as $\mathbb{Q}[G]$-modules. If we can replace $\mathbb{Q}$ by $\mathbb{Z}_{(p)}$, call it a $\mathbb{Z}_{(p)}$-relation.
\end{definition}

Brauer relations induce isogenies between products of Weil restrictions of elliptic curves (\cite{DD10} Theorem 2.3). By considering the degree of the induced isogeny, we will show that the Tate--Shafarevich groups of two abelian varieties differ only in their $p$-parts for $p|n$, and relate these back by $4^{th}$ powers to the terms we want.

\begin{lemma}\label{BrauerRel}
    In $D_{2n}$, there are Brauer relations
    \begin{itemize}
        \item $1 + 2D_{2n} -2C_2 -C_n$ when $n$ is odd, and
        \item $1 + 2D_{2n} -2C_2 -C_n + D_n - D_n'$ when $n$ is even.
    \end{itemize}
    Where, if $D_{2n} = \langle r, s | r^n = s^2 = 1, srs = r^{-1} \rangle $, we let $C_2 = \langle s \rangle$, $C_n = \langle r \rangle$, $D_n = \langle r^2, s \rangle$ and $D_n' = \langle r^2, sr \rangle$.
\end{lemma}
\begin{proof}
    Note that $\mathbb{C}[G/H] \cong \mathrm{Ind}^G_H \mathbf{1}$. For $H = 1$, we get the regular representation, and for $H=G$ we get the trivial representation.

    In the odd case, the irreducible representations are the trivial representation $\mathbf{1}$, the sign representation $\epsilon$, and $\frac{n-1}{2}$ two-dimensional representations.

    We compute $\mathrm{Ind}^{D_{2n}}_{C_2} \mathbf{1}$ using Frobenius reciprocity. For each irreducible representation $\phi$ of $D_{2n}$, we have $\langle \mathrm{Ind}^{D_{2n}}_{C_2} \mathbf{1}, \phi \rangle = \langle \mathbf{1}, \mathrm{Res}^{D_{2n}}_{C_2} \phi \rangle$. We find that the trivial representation appears in $\mathrm{Ind}^{D_{2n}}_{C_2} \mathbf{1}$ once, the sign representation does not appear, and each two-dimensional representation appears once. Similarly we find that $\mathrm{Ind}^{D_{2n}}_{C_2} \mathbf{1} = \mathbf{1} \oplus \epsilon$. We conclude that $\mathbb{C}[D_{2n}/1] \oplus \mathbb{C}[D_{2n}/D_{2n}]^{\oplus 2} \cong \mathbb{C}[D_{2n}/C_n] \oplus \mathbb{C}[D_{2n}/C_2]^{\oplus 2}$.

    Now as these representations are realisable over $\mathbb{Q}$ and isomorphic over $\mathbb{C}$, they are isomorphic over $\mathbb{Q}$ (\cite{Serre} Ch. 12, Prop 33 and discussion). Therefore we have the desired Brauer relation.

    The proof for the even case proceeds similarly. We now have three non-trivial one-dimensional representations, $\epsilon_1$, $\epsilon_2$ and $\epsilon_3$, with kernels $C_n$, $D_n$ and $D_n'$ respectively.

    We find the induced representation from $C_2$ has irreducible summands $\mathbf{1}$, $\epsilon_2$ and all the two-dimensional representations. Inducing from each group of order $n$, we get $\mathbf{1} \oplus \epsilon_i$ for the $\epsilon_i$ corresponding to the subgroup. Putting this together and proceeding as before, we get the desired Brauer relation.
\end{proof}

\begin{lemma}\label{isogeny}
    For the Brauer relations $\sum_i H_i - \sum_j H_j'$ given in Lemma \ref{BrauerRel}, there are isogenies between $A = \prod_i W^{F^{H_i}}_k (E)$ and $B = \prod_j W^{F^{H_j'}}_k (E)$ of degree coprime to $p$, where $p$ is any prime not dividing $n$.
\end{lemma}
\begin{proof}
     By (\cite{Bartel1}, proof of Lemma 3.9), this is a $\mathbb{Z}_{(p)}$-relation. By (\cite{Bartel1}, discussion at the start of Section 3), this is equivalent to saying there is an injection of $\mathbb{Z}D_{2n}$-lattices, $\mathbb{Z}[S_1] \rightarrow \mathbb{Z}[S_2]$ with finite cokernel of order $d$, with $p \nmid d$. Here $S_1 = \bigsqcup_i D_{2n}/H_i$ and $S_2 = \bigsqcup_j D_{2n}/H_j'$. Finally this induces an isogeny of degree $d^2$ as required (stated explicitly in \cite{Bartel2} \S4, based on \cite{DD10} \S4.2 and \cite{Milne} Proposition 6).
\end{proof}
\begin{lemma}\label{Sha}
    Let $A$, $B$ and $p$ be as in Lemma \ref{isogeny}. Then $|\Sha(A/k)[p^{\infty}]| = |\Sha(B/k)[p^{\infty}]|$.
\end{lemma}
\begin{proof}
    As we have an isogeny $\phi : A \rightarrow B$ of degree $d^2$, it has conjugates $\phi'$ and $\phi''$ satisfying $\phi' \circ \phi = \phi \circ \phi'' = [d^2]$. The isogeny $\phi$ induces a homomorphism $\Sha(A/k)[p^{\infty}] \rightarrow \Sha(B/k)[p^{\infty}]$, which is an isomorphism because $[d^2]$ is.
\end{proof}
\begin{proof}[Proof of Theorem \ref{maintheorem}]
    In the case where $n$ is odd, Lemma \ref{Sha} tells us that $$|\Sha(E/F)[p^{\infty}]||\Sha(E/k)[p^{\infty}]|^2 = |\Sha(E/K)[p^{\infty}]||\Sha(E/F^{C_2})[p^{\infty}]|^2.$$ When $\Sha$ of an elliptic curve is finite, it has square order, so $|\Sha(E/F)[p^{\infty}]| \equiv |\Sha(E/K)[p^{\infty}]| \pmod{\mathbb{Q}^{*4}}$.

    In the case where $n$ is even, the same argument tells us $$|\Sha(E/F)[p^{\infty}]||\Sha(E/F^{D_{2n}})[p^{\infty}]| \equiv |\Sha(E/F^{D_{2n'}})[p^{\infty}]||\Sha(E/F^{C_n})[p^{\infty}]| \pmod{\mathbb{Q}^{*4}}.$$ These fixed fields are the three quadratic subextensions of $F/k$, so we can write this as $$|\Sha(E/F)[p^{\infty}]| \equiv |\Sha(E/K_1)[p^{\infty}]||\Sha(E/K_2)[p^{\infty}]||\Sha(E/K_3)[p^{\infty}]| \pmod{\mathbb{Q}^{*4}},$$ again using the squareness of the order of $\Sha$.
\end{proof}

\section*{Galois Module Structure}
In some cases, the same result follows from the work of Chetty, which also gives information about the Galois module structure of $\Sha$. Let $n, p$ and $H \cong C_n$ be as in Theorem \ref{theorem2}, and assume that $\Sha(E/F)[p^{\infty}]$ is finite.

\begin{lemma}\label{localrings}
    $\mathbb{Z}_p[C_n]$ is a direct sum of local rings with principal maximal ideals.
\end{lemma}
\begin{proof}
     $\mathbb{Z}_p[C_n]$ is isomorphic to $\mathbb{Z}_p[T]/(T^n-1)$. This is a direct sum of rings $\mathbb{Z}_p[T]/\Phi_d(T)$ for cyclotomic polynomials $\Phi_d$ with $d|n$. These further split into direct sums because $\Phi_d = \prod P_{d,i}$, a product of irreducible polynomials over $\mathbb{Z}_p$ of degree $\mathrm{ord}_d(p)$. Finally $\mathbb{Z}_p[T]/P_{d,i}(T)$ is the ring of integers of $\mathbb{Q}_p[T]/P_{d,i}(T)$ (\cite{Serre2}, Ch. IV, \S 4, Prop. 16) so has principal maximal ideal.
\end{proof}
\begin{theorem}[ = Theorem \ref{theorem2}]\label{directsum}
    $\Sha(E/F)[p^{\infty}] \cong X \oplus X$ as $\mathbb{Z}_p[H]$-modules, for some submodule $X$, and $\Sha(E/K)[p^{\infty}] \cong X^H \oplus X^H$.

    Moreover, when $n$ is odd and $p^a \equiv -1 \mod{n}$ has a solution, $|X|/|X^H|$ is a square.
\end{theorem}
\begin{proof}
    Apply Lemma 2.8 from \cite{Chetty} to the pairing $(x,sy)$, valued in $\mathbb{Q}_p/\mathbb{Z}_p$, where $s$ is a lift of the non-trivial element of $\mathrm{Gal}(K/k)$ to $\mathrm{Gal}(F/k)$, and $(\_,\_)$ is the Cassels--Tate pairing on $\Sha(E/K)$. This satisfies the required properties because the Cassels--Tate pairing is Galois equivariant, so $\Sha(E/F)[p^{\infty}] \cong X \oplus X$.

    Note that $\Sha(E/F)[p^{\infty}]^H \cong \Sha(E/K)[p^{\infty}]$ as $p \nmid n$ (e.g. \cite{Park}, Lemma 11).

    Now by Lemma \ref{localrings}, $\Sha(E/F)[p^{\infty}]$ is a sum of $\mathbb{Z}_p(T)/P_{d,i}(T)$ modules.
    Suppose that $p^a \equiv -1 \pmod{n}$ has a solution. We will show that $|M|/|M^{C_n}|$ is a square for all $\mathbb{Z}_p[T]/P_{d,i}(T)$-modules $M$, from which it follows that $|X|/|X^H|$ is a square.
    This condition implies $\mathrm{ord}_p(d)$ is even for all $d|n$, $d \neq 1$, and so the degree of $P_{d,i}$ is even. For a uniformiser $\pi$, $\mathbb{Z}_p[T]/P_{d,i}(T)$-modules are of the form $\mathbb{Z}_p[T]/(\pi^a, P_{d,i}(T))$ for some $a$, and have size $p^{a \mathrm{deg}(P_{d,i})}$ as the extension is unramified. These modules have square order as $\mathrm{deg}(P_{d,i})$ is even, and no fixed part.

    When $d=1$, we have $\mathbb{Z}_p$-modules with $C_n$ acting trivially, so $|M|/|M^{C_n}|=1$.
\end{proof}

\begin{remark} In the case where $p^a \equiv -1 \pmod{n}$ doesn't have solutions, we have some $\mathbb{Z}_p[H]$-modules $M_i$ with size an odd power of $p$, with no fixed part. By Theorem \ref{maintheorem} we can conclude that if each $M_i$ appears with multiplicity $2a_i$, then $\sum a_i$ is even.
\end{remark}

\end{document}